\documentclass[portuges,12pt,letter]{article}
\usepackage[centertags]{amsmath}
\usepackage{amsfonts}
\usepackage{newlfont}
\usepackage{amscd}
\usepackage{graphics}
\usepackage{epsfig}
\usepackage{indentfirst}
\usepackage{amsxtra}
\usepackage[latin1]{inputenc}
\usepackage{amssymb, amsmath}
\usepackage{amsthm}
\usepackage[mathscr]{eucal}

\newtheorem{thm}{Theorem}[section]

\newtheorem{remark}[thm]{Remark}

\title{\bf  On General Duality Principles for  Non-Convex Variational  Optimization with
Applications to the Ginzburg-Landau System in Superconductivity}

\author{Fabio Silva Botelho \\
Federal University of Santa Catarina \\
Florian\'{o}polis - SC, Brazil}
\begin{document}
\maketitle

\abstract{This article develops  duality principles applicable to non-convex models in the calculus of variations. The results here developed are applied to
 Ginzburg-Landau type equations. For the first and second duality principles, through an optimality criterion developed for the dual formulations, we qualitatively classify the critical points of the primal and dual functionals in question. We formally prove there is no duality gap between the primal and dual formulations in a local extremal context.

Finally, in the last sections, we present a global existence result, a duality principle and respective optimality conditions for the complex Ginzburg-Landau system in superconductivity in the presence of a magnetic field and concerning magnetic potential. }

\section{Introduction}

In this work we present three theorems which represent duality principles suitable for a large class of non-convex variational problems.

At this point we refer to the exceptionally important article "A contribution to contact problems for a class of solids and structures" by
W.R. Bielski and J.J. Telega, \cite{85},  published in 1985,  as the first one to successfully  apply and generalize the convex analysis approach to a model in non-convex and non-linear mechanics.

The present work is, in some sense, a kind of extension  of this previous work \cite{85} and others such as \cite{2900}, which greatly influenced and
inspired my work and recent book \cite{12}.

 We extend and generalize the  approaches  in \cite{85} and \cite{12a} and develop two multi-duality principles through which we classify qualitatively the critical points.

 Thus, we emphasize the first multi-duality principle generalizes some Toland results found in \cite{12a} and in some appropriate sense, such a work complements the results presented in \cite{85}, now applied to  a Ginzburg-Landau type  model context.

On the other hand, the conclusions of the second multi-duality principle may be qualitatively found in similar form in the triality approach found in \cite{17} and other references
therein, even though the construction of the present result and the respective  proofs be substantially different.

About the model in physics involved, we recall that about the year 1950 Ginzburg and Landau introduced a theory to model the super-conducting behavior of some types of materials below a critical temperature $T_c$,
which depends on the material in question. They postulated the free density energy may be written close to $T_c$ as
$$F_s(T)=F_n(T)+\frac{\hbar}{4m}\int_\Omega |\nabla \psi|^2_2 \;dx+\frac{\alpha(T)}{4}\int_\Omega |\psi|^4\;dx-\frac{\beta(T)}{2}\int_\Omega |\psi|^2\;dx,$$
where $\psi$ is a complex parameter, $F_n(T)$ and $F_s(T)$ are the normal and super-conducting free energy densities, respectively (see \cite{100, 101}
 for details).
Here $\Omega \subset \mathbb{R}^3$ denotes the super-conducting sample with a boundary denoted by $\partial \Omega=\Gamma.$ The complex function $\psi \in W^{1,2}(\Omega; \mathbb{C})$ is intended to minimize
$F_s(T)$ for a fixed temperature $T$.

Denoting $\alpha(T)$ and $\beta(T)$ simply by $\alpha$ and $\beta$,  the corresponding Euler-Lagrange equations are given by:
\begin{equation} \left\{
\begin{array}{ll}
 -\frac{\hbar}{2m}\nabla^2 \psi+\alpha|\psi|^2\psi-\beta\psi=0, & \text{ in } \Omega
 \\ \\
 \frac{\partial {\psi}}{\partial \textbf{n}}=0, &\text{ on } \partial\Omega.\end{array} \right.\end{equation}
This last system of equations is well known as the Ginzburg-Landau (G-L) one in the absence of a magnetic field and respective potential.

\begin{remark}
 About the notation, for an open bounded subset $\Omega \subset \mathbb{R}^3$, we denote the $L^2(\Omega)$ norm by $\| \cdot \|_{L^2(\Omega)}$ or simply by $\| \cdot \|_2$.
Similar remark is valid for the $L^2(\Omega;\mathbb{R}^3)$ norm, which is denoted by $\|\cdot\|_{L^2(\Omega;\mathbb{R}^3)}$ or simply by $\| \cdot \|_2$, when its meaning is clear.
 On the other hand, by $| \cdot |_2$ we denote the standard Euclidean norm in $\mathbb{R}^3$ or $\mathbb{C}^3$.

Moreover derivatives are always understood in the distributional sense. Also, by a regular Lipschitzian boundary $\partial \Omega=\Gamma$ of $\Omega$, we mean regularity enough  so that the standard Sobolev Imbedding  theorems, the trace theorem and Gauss-Green formulas of integration by parts to hold. Details about such results may be found in \cite{1}.

Finally, in general $\delta F(u,v)$ will denote the Fr\'{e}chet derivative of the functional $F(u,v)$ at $(u,v)$,  $$\delta_u F(u,v) \text{ or }\frac{\partial F(u,v)}{\partial u}$$ denotes the first Fr\'{e}chet derivative of $F$ relating the variable $u$
and $$\delta^2 F_{u,v}(u,v) \text{ or } \frac{\partial^2 F(u,v)}{\partial u \partial v}$$ denotes the second one relating the variables $u$ and $v$, always at $(u,v).$

At some point of our analysis, we shall assume a finite-dimensional approximate model version. In such a context, we remark that generically in a matrix sense the notation $$\frac{1}{K+\gamma \nabla^2}$$ will indicate the inverse $$(K I_d+\gamma \nabla^2)^{-1},$$
where $I_d$ denotes the identity matrix and $\nabla^2$ is the matrix originated by a discretized version of the Laplace operator. We also emphasize that,
as the meaning is clear, other similar notations may be used to indicate the inverse of matrices or operators.
\end{remark}

\begin{remark} For an appropriate set $\Omega \subset \mathbb{R}^3$ and a space $U$, our primal functional $J:U \rightarrow \mathbb{R}$ is specified by
$$J(u)= \frac{\gamma}{2}\int_\Omega \nabla u \cdot \nabla u\;dx + \frac{\alpha}{2}\int_\Omega (u^2-\beta)^2\;dx - \langle u,f \rangle_{L^2}$$
$\forall u \in U,$ where $\alpha,\beta, \gamma >0$, $f \in L^2(\Omega)$.

We define,
$$F(u)=-\frac{\gamma}{2}\int_\Omega \nabla u\cdot \nabla u\;dx+\frac{K}{2}\int_\Omega u^2\;dx,$$
where in a finite dimensional discretized model version, in a finite elements or finite differences context, $K>0$ is such that $$F(u)>0,\; \forall u \in U \text{ such that } u \neq \mathbf{0},$$
and
$$G(u,v)=\frac{\alpha}{2}\int_\Omega (u^2-\beta+v)^2\;dx+\frac{K}{2}\int_\Omega u^2\;dx-\langle u,f \rangle_{L^2},$$
so that
$$J(u)=G(u,0)-F(u),\; \forall u \in U.$$

Let $F^*:Y^* \rightarrow \mathbb{R}$ denote the polar functional related to $F$, that is,
$$F^*(v_1^*)=\sup_{u \in U}\{\langle u,v_1^*\rangle_{L^2}-F(u)\}.$$

Since the optimization in question is quadratic, we have
\begin{equation}\label{pw1}F^*(v_1^*)=\langle \tilde{u},v_1^*\rangle_{L^2}-F(\tilde{u}),\end{equation} where $\tilde{u} \in U$ is such that
$$v_1^*=\frac{\partial F(\tilde{u})}{\partial u},$$
that is,
$$v_1^*=\gamma \nabla^2 \tilde{u}+K \tilde{u},$$
so that
$$\tilde{u}=\frac{v_1^*}{K+\gamma \nabla^2}.$$
Replacing such a $\tilde{u}$ into (\ref{pw1}), we obtain

$$F^*(v_1^*)=\frac{1}{2}\int_\Omega v_1^*[(K+\gamma\nabla^2)^{-1}v_1^*]\;dx.$$

Similarly for $G:U \times Y \rightarrow \mathbb{R}$, for $v_0^* \in Y^*$ such that $$2v_0^*+K>0, \text{ in } \overline{\Omega},$$ we also define

$$G^*(v_1^*,v_0^*)=\sup_{(u,v) \in U \times Y}\{ \langle u,v_1^* \rangle_{L^2}+\langle v,v_0^* \rangle_{L^2}-G(u,v)\},$$
where, as above indicated
$$G(u,v)=\frac{\alpha}{2}\int_\Omega (u^2-\beta+v)^2\;dx+\frac{K}{2}\int_\Omega u^2\;dx-\langle u,f \rangle_{L^2}.$$

Defining $w=u^2-\beta+v,$ so that $$v=w-u^2+\beta,$$ we may write
\begin{eqnarray}G^*(v_1^*,v_0^*)&=&\sup_{(u,w) \in U \times Y}\{ \langle u,v_1^* \rangle_{L^2}+\langle w-u^2+\beta,v_0^* \rangle_{L^2}
\nonumber \\ &&-\frac{\alpha}{2}\int_\Omega w^2\;dx-\frac{K}{2} \int_\Omega u^2\;dx+ \langle u,f\rangle_{L^2}\}.
\end{eqnarray}

Since, the optimization in question is quadratic, we have
\begin{eqnarray}\label{pt1}G^*(v_1^*,v_0^*)&=& \langle \tilde{u},v_1^* \rangle_{L^2}+\langle \tilde{w}-\tilde{u}^2+\beta,v_0^* \rangle_{L^2}
\nonumber \\ &&-\frac{\alpha}{2}\int_\Omega\tilde{w}^2\;dx-\frac{K}{2} \int_\Omega \tilde{u}^2\;dx+ \langle \tilde{u},f\rangle_{L^2}\}.
\end{eqnarray}
where $\tilde{u} \in U$ and $\tilde{w} \in Y$ are such that
$$v_1^*= 2\tilde{u}v_0^*+K \tilde{u}-f,$$
and
$$v_0^*=\alpha \tilde{w},$$
so that
$$\tilde{u}=\frac{v_1^*+f}{2v_0^*+K},$$ and
$$\tilde{w}= \frac{v_0^*}{\alpha}.$$
Replacing such results into (\ref{pt1}), we get
\begin{eqnarray}G^*(v_1^*,v_0^*)&=& \frac{1}{2}\int_\Omega \frac{(v_1^*+f)^2}{2 v_0^*+K}\;dx+\frac{1}{2 \alpha}\int_\Omega (v_0^*)^2\;dx
\nonumber \\ &&+\beta \int_\Omega v_0^*\;dx.\end{eqnarray}
\end{remark}

\section{The main result}

In this section we develop a first multi-duality principle for a Ginzburg-Landau type system in a simpler real context.

About these models in physics,   we refer again to \cite{100,101}.

In the next lines we develop the main result. At this point we highlight the global optimality condition $-\gamma \nabla^2+2\hat{v}_0^*> \mathbf{0}$
at a critical point was presented in \cite{17}.

\begin{thm} Let $\Omega \subset \mathbb{R}^3$ be an open, bounded, connected set with a regular (Lipschitzian) boundary denoted by
$\partial \Omega$.  Suppose $J:U \rightarrow \mathbb{R}$ is a functional defined by
$$J(u)= \frac{\gamma}{2}\int_\Omega \nabla u \cdot \nabla u\;dx + \frac{\alpha}{2}\int_\Omega (u^2-\beta)^2\;dx - \langle u,f \rangle_{L^2}$$
$\forall u \in U,$ where $\alpha,\beta, \gamma >0$, $f \in L^2(\Omega)$ and $U=W_0^{1,2}(\Omega).$

 At this point, we assume a discretized finite dimensional model version (in a finite elements or finite differences context, so that from now on, the not relabeled spaces, functions and operators refer to such a  finite dimensional approximation)  and suppose $$\delta J(u_0)=\mathbf{0}$$ 0
where in an appropriate matrices sense, we have
$$\delta^2 J(u_0)=-\gamma \nabla^2 +6 \alpha\;\{ u_0(i)^2\}-2\alpha \beta I_d.$$

Here $\{ u_0(i)^2\}$ denotes the diagonal matrix which the diagonal is given by the vector $[u_0(i)^2]$.

Also, from now on, as the meaning is clear,  we shall denote such a second Fr\'{e}chet derivative simply by $$\delta^2 J(u_0)=-\gamma \nabla^2 +6 \alpha u_0^2-2\alpha \beta.$$

Define,
$$F(u)=-\frac{\gamma}{2}\int_\Omega \nabla u\cdot \nabla u\;dx+\frac{K}{2}\int_\Omega u^2\;dx,$$
where $K>0$ is such that $$F(u)>0,\; \forall u \in U \text{ such that } u \neq \mathbf{0},$$
and
$$G(u,v)=\frac{\alpha}{2}\int_\Omega (u^2-\beta+v)^2\;dx+\frac{K}{2}\int_\Omega u^2\;dx-\langle u,f \rangle_{L^2},$$
so that
$$J(u)=G(u,0)-F(u),\; \forall u \in U.$$

Define also $Y=Y^*=L^2(\Omega)$, $F^*:Y^* \rightarrow \mathbb{R}$ by \begin{eqnarray}F^*(v_1^*)&=&\sup_{u \in U}\{\langle u,v_1^* \rangle_{L^2}-F(u)\}
\nonumber \\ &=& \frac{1}{2} \int_\Omega v_1^*[(K+\gamma \nabla^2)^{-1} v_1^*]\;dx,\end{eqnarray}
and $G^*: Y^* \times Y^* \rightarrow \overline{\mathbb{R}}=\mathbb{R}\cup \{+\infty\}$ by
\begin{eqnarray}G^*(v_1^*,v_0^*)&=&\sup_{(v_1,v) \in Y \times Y}\{\langle v,v_0^*\rangle_Y +\langle v_1,v_1^* \rangle_Y-G(v_1,v)\}
\nonumber \\ &=& \frac{1}{2}\int_\Omega \frac{(v_1^*+f)^2}{2 v_0^*+K}\;dx+\frac{1}{2 \alpha}\int_\Omega (v_0^*)^2\;dx
\nonumber \\ &&+\beta \int_\Omega v_0^*\;dx, \end{eqnarray}
if $v_0^* \in A^*=\{v_0^* \in Y^* \;:\; 2 v_0^*+K>0, \text{ in } \overline{\Omega}\}.$

Let $J^*: Y^* \times Y^* \rightarrow \mathbb{R} \cup\{-\infty\}$ be such that $$J^*(v_1^*,v_0^*)=-G^*(v_1^*,v_0^*)+F^*(v_1^*),$$
so that $\tilde{J}^*:Y^* \rightarrow \mathbb{R}$ is expressed by
$$\tilde{J}^*(v_1^*)=\sup_{v_0^* \in A^*} J^*(v_1^*,v_0^*).$$

Define $$\hat{v}_0^*=\alpha(u_0^2-\beta),$$
and
$$\hat{v}_1^*=(2\hat{v}_0^*+K)u_0-f.$$

Suppose $\hat{v}_0^* \in A^*.$

Under such hypotheses,
$$\delta \tilde{J}^*(\hat{v}_1^*)=\mathbf{0},$$
and $$\tilde{J}^*(\hat{v}_1^*)=J(u_0).$$
Moreover, \begin{enumerate}
\item if $\delta^2 J(u_0)> \mathbf{0}$, then
$$\delta^2 \tilde{J}^*(\hat{v}_1^*)>\mathbf{0},$$ so that there exist $r>0$ and $r_1>0$ such that
\begin{eqnarray}
J(u_0)&=& \min_{ u \in B_r(u_0)} J(u) \nonumber \\ &=& \min_{v_1^* \in B_{r_1}(\hat{v}_1^*)} \tilde{J}^*(v_1^*)\nonumber \\ &=&
\tilde{J}^*(\hat{v}_1^*) \nonumber \\ &=& J^*(\hat{v}_1^*,\hat{v}_0^*).
\end{eqnarray}
\item If  $-\gamma \nabla^2+2\hat{v}_0^*>\mathbf{0}$ so that $\delta^2J(u_0)> \mathbf{0}$,  then
defining $$B^*=\{v_0^* \in Y^*\::\; -\gamma \nabla^2+2 v_0^*>\mathbf{0}\},$$ we have $$\delta J_2^*(\hat{v}_1^*)=\mathbf{0},$$
$$\delta^2 J_2^*(\hat{v}_1^*)> \mathbf{0}$$ and
\begin{eqnarray}
J(u_0)&=& \min_{ u \in U} J(u) \nonumber \\ &=& \min_{v_1^* \in Y^*} J_2^*(v_1^*)\nonumber \\ &=&
J_2^*(\hat{v}_1^*) \nonumber \\ &=& J^*(\hat{v}_1^*,\hat{v}_0^*),
\end{eqnarray}
where $$J_2^*(v_1^*)=\sup_{v_0^* \in A^* \cap B^*} J^*(v_1^*,v_0^*).$$

\item  If $\delta^2 J(u_0)< \mathbf{0}$, then
 $$\delta^2 \tilde{J}^*(\hat{v}_1^*)<\mathbf{0},$$ so that there exist $r>0$ and $r_1>0$ such that
\begin{eqnarray}
J(u_0)&=& \max_{ u \in B_r(u_0)} J(u) \nonumber \\ &=& \max_{v_1^* \in B_{r_1}(\hat{v}_1^*)} \tilde{J}^*(v_1^*)\nonumber \\ &=&
\tilde{J}^*(\hat{v}_1^*) \nonumber \\ &=& J^*(\hat{v}_1^*,\hat{v}_0^*).
\end{eqnarray}
\end{enumerate}
\end{thm}
\begin{proof} From $\delta J(u_0)=\mathbf{0}$ we have

$$-\gamma \nabla^2 u_0+\alpha(u_0^2-\beta)2 u_0-f=0, \text{ in }\Omega$$
so that
\begin{eqnarray}\label{l55}\gamma \nabla^2 u_0+Ku_0&=&\alpha(u_0^2-\beta)2u_0+Ku_0-f \nonumber \\ &=& (2\hat{v}_0^*+K)u_0-f \nonumber \\ &=& \hat{v}_1^*.\end{eqnarray}

From this, we obtain
$$\frac{\hat{v}_1^*}{K+\gamma \nabla^2}-\frac{\hat{v}_1^*+f}{2\hat{v}_0^*+K}= u_0-u_0=0, \text{ in } \Omega.$$

On the other hand,
$$\hat{v}_0^*=\alpha(u_0^2-\beta)=\alpha(u_0^2-\beta+0),$$ so that
$$-\frac{\hat{v}_0^*}{\alpha}+\frac{(\hat{v}_1^*+f)^2}{(2\hat{v}_0^*+K)^2}-\beta=0,$$ and thus
$$\frac{\partial J^*(\hat{v}_1^*, \hat{v}_0^*)}{\partial v_0^*}=0.$$

From this, (\ref{l55}), from the definition of $\hat{v}_0^*$ and the concavity of $J^*(\hat{v}_1^*,v_0^*)$ in $v_0^*$, we obtain
\begin{eqnarray}\tilde{J}^*(\hat{v}_1^*)&=&\sup_{v_0^* \in A^*} J^*(\hat{v}_1^*,v_0^*)
\nonumber \\ &=& J^*(\hat{v}_1^*,\hat{v}_0^*) \nonumber \\ &=& -G^*(\hat{v}_1^*,\hat{v}_0^*)+F^*(\hat{v}_1^*) \nonumber \\ &=&
-\langle u_0,\hat{v}_1^* \rangle_{L^2}-\langle 0,\hat{v}_0^*\rangle_{L^2}+ G(u_0,0)
\nonumber \\ &&+\langle u_0,\hat{v}_1^* \rangle_{L^2}-F(u_0) \nonumber \\ &=& G(u_0,0)-F(u_0)\nonumber \\ &=& J(u_0).\end{eqnarray}

Also, for $v_1^*$ in a neighborhood of $\hat{v}_1^*$ we have that $$J^*(v_1^*)=\sup_{v_0^* \in A^*} J^*(v_1^*,v_0^*)=J^*(v_1^*,\tilde{v}_0^*),$$
where such a supremum is attained through the equation,
$$\frac{\partial J^*(v_1^*,\tilde{v}_0^*)}{\partial v_0^*}= \mathbf{0},$$
so that from the implicit function theorem we have
\begin{eqnarray}\label{h560}\frac{\partial \tilde{J}^*(v_1^*)}{\partial v_1^*}&=&\frac{\partial J^*(v_1^*,\tilde{v}_0^*)}{\partial v_1^*}+\frac{\partial J^*(v_1^*,\tilde{v}_0^*)}{\partial v_0^*}\frac{\partial \tilde{v}_0^*}{\partial v_1^*} \nonumber \\ &=&\frac{\partial J^*(v_1^*,\tilde{v}_0^*)}{\partial v_1^*}.\end{eqnarray}

Moreover, from this, joining the pieces, we get
\begin{eqnarray}
\frac{\partial \tilde{J}^*(\hat{v}_1^*)}{\partial v_1^*}&=& \frac{\partial J^*(\hat{v}_1^*,\hat{v}_0^*)}{\partial v_1^*}+
\frac{\partial J^*(\hat{v}_1^*,\hat{v}_0^*)}{\partial v_0^*} \frac{\partial \hat{v}_0^*}{\partial v_1^*} \nonumber \\ &=&
\frac{\partial J^*(\hat{v}_1^*,\hat{v}_0^*)}{\partial v_1^*} \nonumber \\ &=&
\frac{\hat{v}_1^*}{K+\gamma \nabla^2}-\frac{\hat{v}_1^*+f}{2\hat{v}_0^*+K}= u_0-u_0=0, \text{ in } \Omega.
\end{eqnarray}

Hence, from this and (\ref{h560}), we obtain
\begin{eqnarray}\label{l17}\frac{\partial^2 \tilde{J}^*(\hat{v}_1^*)}{\partial (v_1^*)^2}&=&
\frac{\partial^2 J^*(\hat{v}_1^*,\hat{v}_0^*)}{\partial v_0^*\partial v_1^*} \frac{\partial \hat{v}_0^*}{\partial v_1^*} \nonumber \\ &&
+\frac{\partial^2 J^*(\hat{v}_1^*,\hat{v}_0^*)}{\partial (v_1^*)^2}
\nonumber \\ &=& \frac{2(\hat{v}_1^*+f)}{(2\hat{v}_0^*+K)^2} \frac{\partial \hat{v}_0^*}{\partial v_1^*}
\nonumber \\ &&+\frac{1}{K+\gamma \nabla^2}-\frac{1}{2\hat{v}_0^*+K}. \end{eqnarray}

At this point we may observe that
$$\frac{(\hat{v}_1^*+f)^2}{(2\hat{v}_0^*+K)^2}-\frac{\hat{v}_0^*}{\alpha}-\beta=0,$$
so that taking the variation in $v_1^*$ of this equation in both sides, we obtain
\begin{eqnarray}
&& \frac{2(\hat{v}_1^*+f)}{(2\hat{v}_0^*+K)^2}-\frac{4(\hat{v}_1^*+f)^2}{(2\hat{v}_0^*+K)^3}\frac{\partial \hat{v}_0^*}{\partial v_1^*}
\nonumber \\ &&-\frac{1}{\alpha}\frac{\partial \hat{v}_0^*}{\partial v_1^*}=0,
\end{eqnarray}
and thus,
\begin{equation}\label{l27}\frac{\partial \hat{v}_0^*}{\partial v_1^*}=\frac{\frac{2(\hat{v}_1^*+f)}{(2\hat{v}_0^*+K)^2}}{\frac{1}{\alpha}+\frac{4(\hat{v}_1^*+f)^2}{(2\hat{v}_0^*+K)^3}},
\end{equation}

Replacing (\ref{l27}) into (\ref{l17}), we obtain
\begin{eqnarray}\frac{\partial^2 \tilde{J}^*(\hat{v}_1^*)}{\partial (v_1^*)^2}&=&
\frac{\partial^2 J^*(\hat{v}_1^*,\hat{v}_0^*)}{\partial v_0^*\partial v_1^*} \frac{\partial \hat{v}_0^*}{\partial v_1^*} \nonumber \\ &&
+\frac{\partial^2 J^*(\hat{v}_1^*,\hat{v}_0^*)}{\partial (v_1^*)^2}
\nonumber \\ &=& \frac{2(\hat{v}_1^*+f)}{(2\hat{v}_0^*+K)^2} \frac{\partial \hat{v}_0^*}{\partial v_1^*}
\nonumber \\ &&+\frac{1}{K+\gamma \nabla^2}-\frac{1}{2\hat{v}_0^*+K}
\nonumber \\ &=&
\frac{\frac{4 \alpha (\hat{v}_1^*+f)^2}{(2\hat{v}_0^*+K)^4}}{1+\alpha \frac{4(\hat{v}_1^*+f)^2}{(2\hat{v}_0^*+K)^3}}
\nonumber \\ &&+\frac{1}{K+\gamma \nabla^2}-\frac{1}{2\hat{v}_0^*+K},
\end{eqnarray}

Therefore, denoting $$H=1+\alpha \frac{4(\hat{v}_1^*+f)^2}{(2\hat{v}_0^*+K)^3}=1+\frac{4\alpha u_0^2}{2\hat{v}_0^*+K},$$ we have
\begin{eqnarray}&&\frac{\partial^2 \tilde{J}^*(\hat{v}_1^*)}{\partial (v_1^*)^2}\nonumber \\ &=&\left[\left(1+\alpha \frac{4(\hat{v}_1^*+f)^2}{(2\hat{v}_0^*+K)^3}\right)\frac{1}{K+\gamma \nabla^2} -\frac{1}{2\hat{v}_0^*+K}\right]/H
\nonumber \\ &=& \left[\left(1+\frac{4\alpha u_0^2}{2\hat{v}_0^*+K}\right)\frac{1}{K+\gamma \nabla^2}-\frac{1}{2\hat{v}_0^*+K}\right]/H
\nonumber \\ &=&(2\hat{v}_0^*+K+4\alpha u_0^2-K-\gamma \nabla^2)/[(K+\gamma \nabla^2)(\hat{v}_0^*+K)H] \nonumber
\\ &=& [-\gamma \nabla^2 +6\alpha u_0^2 -2\alpha \beta]/[(K+\gamma \nabla^2)(\hat{v}_0^*+K)H] \nonumber \\ &=& \frac{\delta^2 J(u_0)}{[(K+\gamma \nabla^2)(\hat{v}_0^*+K)H]}.
\end{eqnarray}

Summarizing, assuming $\delta^2J(u_0)> \mathbf{0},$ we obtain
$$\frac{\partial^2 \tilde{J}^*(\hat{v}_1^*)}{\partial (v_1^*)^2}>\mathbf{0},$$
so that $u_0 \in U$ is a point of local minimum for $J$ and $\hat{v}_1^*$ is a point of local minimum for $\tilde{J}^*$.

Hence,
there exist $r>0$ and $r_1>0$ such that
\begin{eqnarray}
J(u_0)&=& \min_{ u \in B_r(u_0)} J(u) \nonumber \\ &=& \min_{v_1^* \in B_{r_1}(\hat{v}_1^*)} \tilde{J}^*(v_1^*)\nonumber \\ &=&
\tilde{J}^*(\hat{v}_1^*) \nonumber \\ &=& J^*(\hat{v}_1^*,\hat{v}_0^*).
\end{eqnarray}

Assume now  $-\gamma \nabla^2+2\hat{v}_0^*> \mathbf{0},$ so that $\delta^2J(u_0)>\mathbf{0}$.

Observe that if $v_0^* \in A^*$ and $-\gamma \nabla^2 +2v_0^*> \mathbf{0},$ then $$\frac{\partial^2 J^*(v_1,v_0^*)}{\partial (v_1^*)^2}=\frac{1}{K+\gamma \nabla^2}-\frac{1}{2v_0^*+K}
> \mathbf{0},$$ so that $J^*(v_1^*,v_0^*)$ is convex in $v_1^*$, $\forall v_0^* \in A^* \cap B^*.$

Hence $$J_2^*(v_1^*)=\sup_{v_0^* \in A^* \cap B^*} J^*(v_1^*,v_0^*),$$ is convex, as the point-wise supremum of a family of convex functionals.

As above, from $\delta J(u_0)=\mathbf{0},$ we may obtain $\delta J_2^*(\hat{v}_1^*)=\mathbf{0},$ so that, since $J_2^*$ is convex, we may infer that
$$J(u_0)=J^*_2(\hat{v}_1^*)=\min_{v_1^* \in Y^*} J^*_2(v_1^*).$$

Moreover,
\begin{eqnarray}
J_2^*(\hat{v}_1^*)&=& \min_{v_1^* \in Y^*} J_2^*(v_1^*)\nonumber \\ &\leq& J_2^*(v_1^*) \nonumber \\ &=& \sup_{v_0^* \in A^* \cap B^*} J^*(v_1^*,v_0^*)
\nonumber \\ &=& \sup_{v_0^* \in A^* \cap B^*} \left\{ \frac{1}{2} \int_\Omega \frac{(v_1^*)^2}{(K+\gamma \nabla^2)}\;dx\right. \nonumber \\ &&\left.-
\frac{1}{2}\int_\Omega \frac{(v_1^*+f)^2}{2v_0^*+K}\;dx-\frac{1}{2} \int_\Omega \frac{(v_0^*)^2}{\alpha}\;dx-\beta \int_\Omega v_0^*\;dx\right\} \nonumber \\ &\leq&
\sup_{v_0^* \in A^* \cap B^*} \left\{\frac{1}{2} \int_\Omega \frac{(v_1^*)^2}{(K+\gamma \nabla^2)}\;dx\right. \nonumber \\ &&\left.-
\langle u,v_1^* +f\rangle_{L^2}+\int_\Omega(2v_0^*+K) \frac{u^2}{2}\;dx-\frac{1}{2} \int_\Omega \frac{(v_0^*)^2}{\alpha}\;dx-\beta \int_\Omega v_0^*\;dx\right\}
\nonumber \\ &\leq& \sup_{v_0^* \in Y^*} \left\{\frac{1}{2} \int_\Omega \frac{(v_1^*)^2}{(K+\gamma \nabla^2)}\;dx\right. \nonumber \\ &&\left.-
\langle u,v_1^*+f \rangle_{L^2}+\int_\Omega(2v_0^*+K) \frac{u^2}{2}\;dx-\frac{1}{2} \int_\Omega \frac{(v_0^*)^2}{\alpha}\;dx-\beta \int_\Omega v_0^*\;dx\right\}
\nonumber \\ &=& \frac{1}{2} \int_\Omega \frac{(v_1^*)^2}{(K+\gamma \nabla^2)}\;dx \nonumber \\ &&-
\langle u,v_1^* \rangle_{L^2}+\frac{\alpha}{2}\int_\Omega (u^2-\beta)^2\;dx+\frac{K}{2}\int_\Omega u^2\;dx\nonumber \\ &&-\langle u,f\rangle_{L^2},
\;\forall u \in U,\;v_1^* \in Y^*.
\end{eqnarray}

Hence \begin{eqnarray}
J_2^*(\hat{v}_1^*)&\leq& \inf_{v_1^* \in Y^*}\left\{\frac{1}{2} \int_\Omega \frac{(v_1^*)^2}{(K+\gamma \nabla^2)}\;dx \right.\nonumber \\ &&\left.-
\langle u,v_1^* \rangle_{L^2}+\frac{\alpha}{2}\int_\Omega (u^2-\beta)^2\;dx+\frac{K}{2}\int_\Omega u^2\;dx-\langle u,f\rangle_{L^2}\right\} \nonumber \\ &=&
\frac{\gamma}{2} \int_\Omega \nabla u \cdot \nabla u\;dx-\frac{K}{2}\int_\Omega u^2\;dx+\frac{\alpha}{2}\int_\Omega (u^2-\beta)^2\;dx+\frac{K}{2}\int_\Omega u^2\;dx-\langle u,f\rangle_{L^2}\nonumber \\ &=& J(u),\; \forall u \in U.\end{eqnarray}

From this and $$J(u_0)=J_2^*(\hat{v}_1^*),$$
we obtain
\begin{eqnarray}
J(u_0)&=& \min_{ u \in U} J(u) \nonumber \\ &=& \min_{v_1^* \in Y^*} J_2^*(v_1^*)\nonumber \\ &=&
J_2^*(\hat{v}_1^*) \nonumber \\ &=& J^*(\hat{v}_1^*,\hat{v}_0^*).
\end{eqnarray}

The third item may be proven similarly as the first one.

This completes the proof.

\end{proof}

\section{ A Second multi-duality principle}

In this section, in a similar context, we present a second multi-duality principle. This principle is significantly different from the previous one,
since we invert the order of variables as evaluating the extremals.

Indeed the first part of this proof is similar  to the one of the previous theorem. Important differences appears along the proof. For the sake of completeness, we present such a proof in details.

Finally, we emphasize the conclusions of this second multi-duality principle may be qualitatively found in similar form in the triality approach found in \cite{17} and other references
therein. 

\begin{thm} Let $\Omega \subset \mathbb{R}^3$ be an open, bounded, connected set with a regular (Lipschitzian) boundary denoted by
$\partial \Omega$.  Suppose $J:U \rightarrow \mathbb{R}$ is a functional defined by
$$J(u)= \frac{\gamma}{2}\int_\Omega \nabla u \cdot \nabla u\;dx + \frac{\alpha}{2}\int_\Omega (u^2-\beta)^2\;dx - \langle u,f \rangle_{L^2}$$
$\forall u \in U,$ where $\alpha,\beta, \gamma >0$, $f \in L^2(\Omega)$ and $U=W_0^{1,2}(\Omega).$

 At this point, we assume a discretized finite dimensional model version (in a finite elements or finite differences context, so that from now on, the not relabeled spaces, functions and operators refer to such a  finite dimensional approximation)  and suppose $$\delta J(u_0)=\mathbf{0}$$  where in an appropriate matrices sense, we have
$$\delta^2 J(u_0)=-\gamma \nabla^2 +6 \alpha\;u_0^2-2\alpha \beta.$$

Define,
$$F(u)=-\frac{\gamma}{2}\int_\Omega \nabla u\cdot \nabla u\;dx+\frac{K}{2}\int_\Omega u^2\;dx,$$
where $K>0$ is such that $$F(u)>0,\; \forall u \in U \text{ such that } u \neq \mathbf{0},$$
and
$$G(u,v)=\frac{\alpha}{2}\int_\Omega (u^2-\beta+v)^2\;dx+\frac{K}{2}\int_\Omega u^2\;dx-\langle u,f \rangle_{L^2},$$
so that
$$J(u)=G(u,0)-F(u),\; \forall u \in U.$$

Define also $Y=Y^*=L^2(\Omega)$, $F^*:Y^* \rightarrow \mathbb{R}$ by \begin{eqnarray}F^*(v_1^*)&=&\sup_{u \in U}\{\langle u,v_1^* \rangle_{L^2}-F(u)\}
\nonumber \\ &=& \frac{1}{2} \int_\Omega v_1^*[(K+\gamma \nabla^2)^{-1} v_1^*]\;dx,\end{eqnarray}
and $G^*: Y^* \times Y^* \rightarrow \overline{\mathbb{R}}=\mathbb{R}\cup \{+\infty\}$ by
\begin{eqnarray}G^*(v_1^*,v_0^*)&=&\sup_{(v_1,v) \in Y \times Y}\{\langle v,v_0^*\rangle_Y +\langle v_1,v_1^* \rangle_Y-G(v_1,v)\}
\nonumber \\ &=& \frac{1}{2}\int_\Omega \frac{(v_1^*+f)^2}{2 v_0^*+K}\;dx+\frac{1}{2 \alpha}\int_\Omega (v_0^*)^2\;dx
\nonumber \\ &&+\beta \int_\Omega v_0^*\;dx, \end{eqnarray}
if $v_0^* \in A^*=\{v_0^* \in Y^* \;:\; v_0^*+K>0, \text{ in } \overline{\Omega}\}.$

Let $J^*: Y^* \times Y^* \rightarrow \mathbb{R} \cup\{-\infty\}$ be such that $$J^*(v_1^*,v_0^*)=-G^*(v_1^*,v_0^*)+F^*(v_1^*).$$

Define $$\hat{v}_0^*=\alpha(u_0^2-\beta),$$
and
$$\hat{v}_1^*=(2\hat{v}_0^*+K)u_0-f.$$

Suppose $\hat{v}_0^* \in A^*.$

Under such hypotheses,
\begin{enumerate}
\item If $\delta^2J(u_0)>\mathbf{0}$ and $-\gamma \nabla^2+2\hat{v}_0^*>\mathbf{0}$, then $$\delta J_1^*(\hat{v}_0^*)=\mathbf{0}$$ and
$$\delta^2 J_1^*(\hat{v}_0^*)=-\frac{\delta^2J(u_0)}{\alpha (-\gamma \nabla^2+2\hat{v}_0^*)}<\mathbf{0},$$
so that there exist $r,r_1>0$ such that
\begin{eqnarray}
J(u_0)&=& \inf_{u \in B_r(u_0)} J(u) \nonumber \\ &=& \sup_{v_0^* \in B_{r_1}(\hat{v}_0^*)} \left\{ \inf_{v_1^* \in Y^*} J^*(v_1^*,v_0^*) \right\}
\nonumber \\ &=& \sup_{v_0^* \in B_{r_1}(\hat{v}_0^*)} J_1^*(v_0^*) \nonumber \\ &=& J_1^*(\hat{v}_0^*) \nonumber \\ &=& J^*(\hat{v}_1^*,\hat{v}_0^*),
\end{eqnarray}
where $$J_1^*(v_0^*)=\inf_{v_1^* \in Y^*} J^*(v_1^*,v_0^*).$$
\item If $\delta^2J(u_0)>\mathbf{0}$ and $-\gamma \nabla^2+2\hat{v}_0^*<\mathbf{0}$, then $$\delta J_2^*(\hat{v}_0^*)=\mathbf{0}$$ and
$$\delta^2 J_2^*(\hat{v}_0^*)=-\frac{\delta^2J(u_0)}{\alpha (-\gamma \nabla^2+2\hat{v}_0^*)}>\mathbf{0},$$
so that there exist $r,r_1>0$ such that
\begin{eqnarray}
J(u_0)&=& \inf_{u \in B_r(u_0)} J(u) \nonumber \\ &=& \inf_{v_0^* \in B_{r_1}(\hat{v}_0^*)} \left\{ \sup_{v_1^* \in Y^*} J^*(v_1^*,v_0^*) \right\}
\nonumber \\ &=& \inf_{v_0^* \in B_{r_1}(\hat{v}_0^*)} J_2^*(v_0^*) \nonumber \\ &=& J_2^*(\hat{v}_0^*) \nonumber \\ &=& J^*(\hat{v}_1^*,\hat{v}_0^*),
\end{eqnarray}
where $$J_2^*(v_0^*)=\sup_{v_1^* \in Y^*} J^*(v_1^*,v_0^*).$$

\item If $\delta^2J(u_0)<\mathbf{0}$ so that $-\gamma \nabla^2+2\hat{v}_0^*<\mathbf{0}$, then $$\delta J_2^*(\hat{v}_1^*)=\mathbf{0}$$ and
$$\delta^2 J_2^*(\hat{v}_0^*)=-\frac{\delta^2J(u_0)}{\alpha (-\gamma \nabla^2+2\hat{v}_0^*)}<\mathbf{0},$$
so that there exist $r,r_1>0$ such that
\begin{eqnarray}
J(u_0)&=& \sup_{u \in B_r(u_0)} J(u) \nonumber \\ &=& \sup_{v_0^* \in B_{r_1}(\hat{v}_0^*)} \left\{ \sup_{v_1^* \in Y^*} J^*(v_1^*,v_0^*) \right\}
\nonumber \\ &=& \sup_{v_0^* \in B_{r_1}(\hat{v}_0^*)} J_2^*(v_0^*) \nonumber \\ &=& J_2^*(\hat{v}_0^*) \nonumber \\ &=& J^*(\hat{v}_1^*,\hat{v}_0^*),
\end{eqnarray}
where $$J_2^*(v_0^*)=\sup_{v_1^* \in Y^*} J^*(v_1^*,v_0^*).$$
\end{enumerate}
\end{thm}
\begin{proof} From $\delta J(u_0)=\mathbf{0}$ we have

$$-\gamma \nabla^2 u_0+\alpha(u_0^2-\beta)2 u_0-f=0, \text{ in }\Omega$$
so that
\begin{eqnarray}\label{l5}\gamma \nabla^2 u_0+Ku_0&=&\alpha(u_0^2-\beta)2u_0+Ku_0-f \nonumber \\ &=& (2\hat{v}_0^*+K)u_0-f \nonumber \\ &=& \hat{v}_1^*.\end{eqnarray}

From this, we obtain
$$\frac{\hat{v}_1^*}{K+\gamma \nabla^2}-\frac{\hat{v}_1^*+f}{2\hat{v}_0^*+K}= u_0-u_0=0, \text{ in } \Omega.$$

Thus,
$$\frac{\partial J^*(\hat{v}_1^*, \hat{v}_0^*)}{\partial v_1^*}=0.$$

Define $$\tilde{J}(v_0^*)=J^*(\tilde{v}_1^*,v_0^*),$$ where $\tilde{v}_1^* \in Y^*$ is such that
$$\frac{\partial J^*(\tilde{v}_1^*, v_0^*)}{\partial v_1^*}=0.$$

Observe that, from $$\frac{\partial^2 J^*(v_1^*,v_0^*)}{\partial (v_1^*)^2}=\frac{1}{K+\gamma \nabla^2}-\frac{1}{2v_0^*+K},$$
we have that, if $$-\gamma \nabla^2+2v_0^*> \mathbf{0},$$ then
$$\tilde{J}(v_0^*)=J^*(\tilde{v}_1^*,v_0^*)=\inf_{v_1^* \in Y^*} J^*(v_1^*,v_0^*),$$
whereas if $$-\gamma \nabla^2+2v_0^*< \mathbf{0},$$ then
$$\tilde{J}(v_0^*)=J^*(\tilde{v}_1^*,v_0^*)=\sup_{v_1^* \in Y^*} J^*(v_1^*,v_0^*).$$

From (\ref{l5}) and from the definition of $\hat{v}_0^*$, we obtain
\begin{eqnarray}\tilde{J}^*(\hat{v}_0^*)&=&
 J^*(\hat{v}_1^*,\hat{v}_0^*) \nonumber \\ &=& -G^*(\hat{v}_1^*,\hat{v}_0^*)+F^*(\hat{v}_1^*) \nonumber \\ &=&
-\langle u_0,\hat{v}_1^* \rangle_{L^2}-\langle 0,\hat{v}_0^*\rangle_{L^2}+ G(u_0,0)
\nonumber \\ &&+\langle u_0,\hat{v}_1^* \rangle_{L^2}-F(u_0) \nonumber \\ &=& G(u_0,0)-F(u_0)\nonumber \\ &=& J(u_0).\end{eqnarray}

From the implicit function theorem we have
\begin{eqnarray}\label{h56}\frac{\partial \tilde{J}^*(v_0^*)}{\partial v_0^*}&=&\frac{\partial J^*(\tilde{v}_1^*,v_0^*)}{\partial v_0^*}+\frac{\partial J^*(\tilde{v}_1^*,v_0^*)}{\partial v_1^*}\frac{\partial \tilde{v}_1^*}{\partial v_0^*} \nonumber \\ &=&\frac{\partial J^*(\tilde{v}_1^*,v_0^*)}{\partial v_0^*}.\end{eqnarray}

Moreover, from this, joining the pieces, we get
\begin{eqnarray}
\frac{\partial \tilde{J}^*(\hat{v}_0^*)}{\partial v_0^*}&=& \frac{\partial J^*(\hat{v}_1^*,\hat{v}_0^*)}{\partial v_0^*}+
\frac{\partial J^*(\hat{v}_1^*,\hat{v}_0^*)}{\partial v_1^*} \frac{\partial \hat{v}_1^*}{\partial v_0^*} \nonumber \\ &=&
\frac{\partial J^*(\hat{v}_1^*,\hat{v}_0^*)}{\partial v_0^*} \nonumber \\ &=&
\frac{(\hat{v}_1^*+f)^2}{(2\hat{v}_0^*+K)^2}-\frac{\hat{v}_0^*}{\alpha}-\beta \nonumber \\ &=& u_0^2-\frac{\hat{v}_0^*}{\alpha}-\beta
\nonumber \\ &=& 0, \text{ in } \Omega,
\end{eqnarray}
since $\hat{v}_0^*=\alpha (u_0^2-\beta), \; \text{ in } \Omega.$

Hence, from this and (\ref{h56}), we obtain
\begin{eqnarray}\label{l1}\frac{\partial^2 \tilde{J}^*(\hat{v}_0^*)}{\partial (v_0^*)^2}&=&
\frac{\partial^2 J^*(\hat{v}_1^*,\hat{v}_0^*)}{\partial v_0^*\partial v_1^*} \frac{\partial \hat{v}_1^*}{\partial v_0^*} \nonumber \\ &&
+\frac{\partial^2 J^*(\hat{v}_1^*,\hat{v}_0^*)}{\partial (v_0^*)^2}
\nonumber \\ &=& \frac{2(\hat{v}_1^*+f)}{(2\hat{v}_0^*+K)^2} \frac{\partial \hat{v}_1^*}{\partial v_0^*}
\nonumber \\ &&-4\frac{(\hat{v}_1^*+f)^2}{(2\hat{v}_0^*+K)^3}-\frac{1}{\alpha}. \end{eqnarray}

At this point we may observe that
$$\frac{\hat{v}_1^*}{K +\gamma \nabla^2}-\frac{(\hat{v}_1^*+f)}{(2\hat{v}_0^*+K)}=0,$$
so that taking the variation in $v_0^*$ of this equation in both sides, we obtain
\begin{eqnarray}
&& \frac{\frac{\partial \hat{v}_1^*}{\partial v_0^*}}{(K+\gamma \nabla^2)}-\frac{\frac{\partial \hat{v}_1^*}{\partial v_0^*}}{(2\hat{v}_0^*+K)}
+\frac{2(\hat{v}_1^*+f)}{(2\hat{v}_0^*+K)^2} = 0,
\end{eqnarray}
and thus, recalling that $$u_0=\frac{(\hat{v}_1^*+f)}{(2\hat{v}_0^*+K)},$$ we get
\begin{equation}\label{l2}\frac{\partial \hat{v}_1^*}{\partial v_0^*}=-\frac{ \frac{ (K+\gamma \nabla^2)(2u_0)}{(2\hat{v}_0^*+K)}} {1-\frac{K+\gamma \nabla^2}{2\hat{v}_0^*+K}},
\end{equation}

Replacing (\ref{l2}) into (\ref{l1}), we obtain
\begin{eqnarray}\frac{\partial^2 \tilde{J}^*(\hat{v}_0^*)}{\partial (v_0^*)^2}&=&
\frac{\partial^2 J^*(\hat{v}_1^*,\hat{v}_0^*)}{\partial v_0^*\partial v_1^*} \frac{\partial \hat{v}_1^*}{\partial v_0^*} \nonumber \\ &&
+\frac{\partial^2 J^*(\hat{v}_1^*,\hat{v}_0^*)}{\partial (v_0^*)^2}
\nonumber \\ &=& \frac{2(\hat{v}_1^*+f)}{(2\hat{v}_0^*+K)^2} \frac{\partial \hat{v}_1^*}{\partial v_0^*}
\nonumber \\ && -\frac{4(\hat{v}_1^*+f)^2}{(2\hat{v}_0^*+K)^3}-\frac{1}{\alpha}
\nonumber \\ &=&
-\frac{ \frac{ K+\gamma \nabla^2}{(2\hat{v}_0^*+K)^2} (4u_0^2)}{1-\frac{K+\gamma \nabla^2}{2\hat{v}_0^*+K}}-
\frac{1}{\alpha}-\frac{4u_0^2}{2\hat{v}_0^*+K}.
\end{eqnarray}
Therefore,
\begin{eqnarray}\label{pl1}&&\frac{\partial^2 \tilde{J}^*(\hat{v}_0^*)}{\partial (v_0^*)^2} \nonumber \\ &=&
-\frac{1}{\alpha}-\frac{4u_0^2}{2\hat{v}_0^*+K} \nonumber \\ &&
- \frac{K+\gamma \nabla^2}{2v_0^*+K} \frac{4u_0^2}{(2\hat{v}_0^*-\gamma \nabla^2)} \nonumber \\
&=& -\frac{1}{\alpha}-\frac{4u_0^2}{(2\hat{v}_0+K)}\left(1+\frac{K+\gamma \nabla^2}{2\hat{v}_0^*-\gamma \nabla^2}\right)
\nonumber \\ &=& -\frac{1}{\alpha}-\frac{4u_0^2}{(2\hat{v}_0^*-\gamma \nabla^2)} \nonumber \\ &=&
\frac{\gamma \nabla^2-2\hat{v}_0^*-4\alpha u_0^2}{\alpha (-\gamma \nabla^2+2\hat{v}_0^*)} \nonumber \\
&=& \frac{\gamma \nabla^2-6\alpha u_0^2+2\alpha\beta}{\alpha (-\gamma \nabla^2+2\hat{v}_0^*)}
\nonumber \\ &=& -\frac{\delta^2J(u_0)}{\alpha (-\gamma \nabla^2+2\hat{v}_0^*)}.
\end{eqnarray}

Assume now $\delta^2J(u_0)>\mathbf{0}$ and $-\gamma \nabla^2+2 \hat{v}_0^*>\mathbf{0}.$

From (\ref{pl1}), we obtain

$$\delta^2J^*_1(\hat{v}_0^*)=\frac{\partial^2 \tilde{J}^*(\hat{v}_0^*)}{\partial (v_0^*)^2}=-\frac{\delta^2J(u_0)}{\alpha (-\gamma \nabla^2+2\hat{v}_0^*)}<\mathbf{0}.$$

Summarizing,
$$\delta^2J^*_1(\hat{v}_0^*)<\mathbf{0},$$
so that $u_0 \in U$ is a point of local minimum for $J$ and $\hat{v}_0^*$ is a point of local maximum for $J_1^*$.

Hence,
there exist $r>0$ and $r_1>0$ such that
\begin{eqnarray}
J(u_0)&=& \min_{ u \in B_r(u_0)} J(u) \nonumber \\ &=& \max_{v_0^* \in B_{r_1}(\hat{v}_0^*)} J_1^*(v_0^*)\nonumber \\ &=&
J_1^*(\hat{v}_0^*) \nonumber \\ &=& J^*(\hat{v}_1^*,\hat{v}_0^*).
\end{eqnarray}

The remaining items may be proven similarly from (\ref{pl1}).

This completes the proof.

\end{proof}

\section{Another duality principle for global optimization}

Our next result is another duality principle suitable for global optimization. The optimality criterion here presented may be found in \cite{17}.
\begin{thm} Let $\Omega \subset \mathbb{R}^3$ be an open, bounded, connected set with a  Lipschitzian boundary denoted by
$\partial \Omega$. Consider the Ginzburg-Landau energy given by
$J:U \rightarrow \mathbb{R}$, where
\begin{eqnarray}J(u)&=&\frac{\gamma}{2}\int_\Omega \nabla u \cdot \nabla u\;dx \nonumber \\
&&+\frac{\alpha}{2}\int_\Omega (u^2-1)^2\;dx-\langle u,f \rangle_{L^2},
\end{eqnarray}
where $\alpha,\;\gamma>0$, $f \in L^2(\Omega)$ and
$$U=W_0^{1,2}(\Omega)=\{u \in W^{1,2}(\Omega)\;:\; u=0, \; \text{ on } \partial \Omega\}.$$

We also denote, for a finite dimensional discretized version of this problem, in a finite elements or finite differences context,

$$J(u)=-F(u)+G_1(u,0),$$
where
$$F(u)=-\frac{\gamma}{2}\int_\Omega \nabla u \cdot \nabla u\;dx+\frac{K}{2}\int_\Omega u^2\;dx,$$
and
$$G_1(u,v)=\frac{\alpha}{2}\int_\Omega(u^2-1+v)^2\;dx+\frac{K}{2}\int_\Omega u^2\;dx-\langle u,f \rangle_{L^2},$$
where
$K>0$ is such that $F(u) > 0, \; \forall u \in U, \text{ such that } u\neq \mathbf{0}.$

And where generically,
$$\langle h,g\rangle_{L^2}=\int_\Omega hg\;dx,\; \forall h,g \in L^2(\Omega).$$

We define,
\begin{eqnarray}F^*(z^*)&=& \sup_{u \in U}\{\langle z^*,u\rangle_{L^2}-F(u)\}
\nonumber \\ &=& \sup_{u \in U}\{\langle z^*,u\rangle_{L^2}+\frac{\gamma}{2}\int_\Omega \nabla u\cdot \nabla u\;dx
\nonumber \\ &&-\frac{K}{2}\int_\Omega u^2\;dx\} \nonumber \\ &=&
\frac{1}{2}\int_\Omega z^*((K I_d+\gamma\nabla^2)^{-1} z^*)\;dx,
\end{eqnarray}
where $I_d$ denotes the identity matrix.

Also,
\begin{eqnarray}
G_1^*(z^*,v_1^*)&=& \sup_{(u,v) \in U \times L^2}\{ \langle z^*,u \rangle_{L^2}+\langle v_1^*,v\rangle_{L^2}
\nonumber \\ &&+\langle u,f \rangle_{L^2}-\frac{\alpha}{2}\int_\Omega (u^2-1+v)^2\;dx-\frac{K}{2}\int_\Omega u^2\;dx\}
\nonumber \\ &=& \frac{1}{2}\int_\Omega \frac{(z^*+f)^2}{2v_1^*+K}\;dx +\frac{1}{2\alpha}\int_\Omega (v_1^*)^2\;dx
\nonumber \\ &&+\int_\Omega v_1^*\;dx \nonumber \\ &\equiv& G_{1L}^*(z^*,v_1^*),
\end{eqnarray}
if $v^*_1 \in B_1,$
where
$$B_1=\{v_1^* \in Y^*\;:\; 2v_1^*+K>0, \text{ in } \overline{\Omega}\}$$
and $G_{1L}^*$ stands for the Legendre transform of $G_1$.

We also denote,
\begin{eqnarray}
B_2&=& \{v_1^* \in Y^* \;:\; \nonumber \\ &&
\frac{\gamma}{2}\int_\Omega \nabla u \cdot \nabla u\;dx+\int_\Omega v_1^* u^2\;dx > 0, \nonumber
\\ && \forall u \in U \text{ such that } u \neq \mathbf{0}\},
\end{eqnarray}
$$C^*=B_1 \cap B_2,$$
where $$Y=Y^*=L^2(\Omega).$$
Under such hypotheses,
\begin{eqnarray}\inf_{u \in U} J(u) &\geq& \sup_{v_1^* \in C^*}\{\inf_{z^* \in Y^*}\{ J^*(z^*,v_1^*)\}\}
\nonumber \\ &=& \sup_{v_1^* \in C^*} \tilde{J}(v_1^*),\end{eqnarray}
where, $$J^*(z^*,v_1^*)=F^*(z^*)-G_1^*(z^*,v_1^*)$$
and
$$\tilde{J}(v_1^*)=\inf_{z^* \in Y^*} J^*(z^*,v_1^*).$$
Moreover, if there exists a critical point $(z_0^*,(v_0^*)_1) \in C^* \times Y^*$, so that
$$\delta J^*(z^*_0,(v_0^*)_1)=0,$$
then, denoting $$u_0=\frac{z_0^*+f}{2(v_0^*)_1+K}$$ we have that
\begin{eqnarray}
J(u_0)&=& \min_{u \in U} J(u) \nonumber \\ &=& \max_{v^*_1 \in C^*} \tilde{J}^*(v^*_1) \nonumber \\ &=& \tilde{J}^*((v_0^*)_1)
\nonumber \\ &=& J^*(z_0^*,(v_0^*)_1).
\end{eqnarray}

\end{thm}
\begin{proof} Observe that
\begin{eqnarray}
&&G_1^*(z^*,v_1^*) \nonumber \\ &\geq&
\langle z^*,u \rangle_{L^2}+\langle v_1^*,v\rangle_{L^2}-G_1(u,v),
\end{eqnarray}
$\forall v_1^* \in C^*,\; u \in U,\; v \in Y,\; z^* \in Y^*.$

Thus,
\begin{eqnarray} &&-\langle z^*,u \rangle_{L^2}+G_1(u,0) \nonumber \\ &\geq&
-G_1^*(z^*,v_1^*), \end{eqnarray}
$\forall u \in U, \; v^*_1 \in C^*,\;z^* \in Y^*,$
so that
\begin{eqnarray} &&F^*(z^*)-\langle z^*,u \rangle_{L^2}+G_1(u,0) \nonumber \\ &\geq&
F^*(z^*)-G_1^*(z^*,v_1^*), \end{eqnarray}
$\forall u \in U, \; v^*_1 \in C^*,\;z^* \in Y^*.$

Hence,
\begin{eqnarray}J(u)&=& -F(u)+G_1(u,0) \nonumber \\ &=&\inf_{z^* \in Y^*}\{ F^*(z^*)-\langle z^*,u \rangle_{L^2}\}+G_1(u,0) \nonumber \\ &\geq&
\inf_{z^* \in Y^*} \{F^*(z^*)-G_1^*(z^*,v_1^*)\}\nonumber \\
&=& \inf_{z^* \in Y^*} J^*(z^*,v_1^*) \nonumber \\ &=& \tilde{J}^*(v_1^*), \end{eqnarray}
$\forall u \in U,\; v_1^* \in C^*.$

Thus,
\begin{equation}\label{us3099}\inf_{u \in U} J(u)\geq \sup_{v^*_1 \in C^*} \tilde{J}^*(v^*_1).\end{equation}

Now suppose $(z_0^*,(v_0^*)_1) \in Y^* \times C^*$ is such that
$$\delta  J^*(z_0^*,(v_0^*)_1)=\mathbf{0}.$$
From the variation in $z^*$ we obtain,
$$(K I_d+\gamma\nabla^2)^{-1}(z_0^*)=\frac{z_0^*+f}{2(v_0^*)_1+K}=u_0,$$
so that
$$z_0^*=(K I_d+\gamma \nabla^2) u_0,$$
and
\begin{equation}\label{br509} z_0^*+f=(2(v_0^*)_1+K)u_0.\end{equation}
Thus,
\begin{equation}\label{us3100}F^*(z_0^*)=\langle z_0^*,u_0 \rangle_{L^2}-F(u_0).\end{equation}

On the other hand, from the variation in $v_1^*$, we have,
$$\frac{[z_0^*+f]^2}{[2(v_0^*)_1+K]^2}-\frac{(v_0^*)_1}{\alpha} -1=0,$$
so that
$$(v_0^*)_1=\alpha(u_0^2-1),$$
and hence, from this and (\ref{br509}) we have,
$$z_0^*+f=\alpha(u_0^2-1) 2u_0+Ku_0,$$ and
\begin{equation}\label{us3180}G_1^*(z_0^*,(v_0^*)_1)=\langle z_0^*,u_0 \rangle_{L^2}-G_1(u_0,0).\end{equation}

From (\ref{us3100}) and (\ref{us3180}), we obtain
\begin{eqnarray}\label{us3190}
J(u_0)&=& -F(u_0)+G_1(u_0,0) \nonumber \\ &=&
F^*(z_0^*)-G_1^*(z_0^*,(v_0^*)_1) \nonumber \\
&=& J^*(z_0^*,(v_0^*)_1).
\end{eqnarray}

Now, let $$v^*_1 \in C^*.$$

Observe that, in such a case,
$$\frac{\gamma}{2}\int_\Omega \nabla u \cdot \nabla u \;dx +\int_\Omega v_1^* u^2\;dx > 0,$$
$\forall u \in U, \text{ such that } u \neq \mathbf{0}.$

Denoting
\begin{eqnarray}
\alpha_1
 = \inf_{u \in U} \left\{\frac{\gamma}{2}\int_\Omega \nabla u \cdot \nabla u \;dx +\int_\Omega v_1^* u^2\;dx
 -\langle u,f \rangle_{L^2} \right\},
 \end{eqnarray}
 we have
\begin{eqnarray}
&& \int_\Omega v_1^* u^2\;dx+\frac{K}{2}\int_\Omega u^2 \;dx
-\langle u,f \rangle_{L^2}-\langle z^*,u \rangle_{L^2}
\nonumber \\ &\geq& -\frac{\gamma}{2}\int_\Omega \nabla u \cdot \nabla u \;dx+\frac{K}{2}\int_\Omega u^2 \;dx -\langle z^*,u \rangle_{L^2}+\alpha_1,
\end{eqnarray}
$\forall u \in U,$ so that,
\begin{eqnarray}
&& \inf_{u \in U}\left\{\int_\Omega v_1^* u^2\;dx+\frac{K}{2}\int_\Omega u^2 \;dx-
\langle u,f \rangle_{L^2}-\langle z^*,u \rangle_{L^2}\right\}
\nonumber \\ &\geq& \inf_{u \in U}\left\{-\frac{\gamma}{2}\int_\Omega \nabla u \cdot \nabla u \;dx+\frac{K}{2}\int_\Omega u^2 \;dx -\langle z^*,u \rangle_{L^2}\right\}+\alpha_1,
\end{eqnarray}
and hence,
$$-\frac{1}{2}\int_\Omega \frac{(z^*+f)^2}{2v_1^*+K} \geq -F^*(z^*)+\alpha_1,$$
so that, for $v_1^* \in C^*$ fixed, we have,
$$F^*(z^*)-\frac{1}{2}\int_\Omega \frac{(z^*+f)^2}{2v_1^*+K}\;dx \geq \alpha_1,$$
$\forall z^* \in Y^*.$

And indeed, from the general result in \cite{12a},
we have
$$\inf_{z^* \in Y^*}\left\{F^*(z^*)-\frac{1}{2}\int_\Omega \frac{(z^*+f)^2}{2v_1^*+K}\;dx\right\} = \alpha_1 \in \mathbb{R}.$$

From this, since $(v_0^*)_1 \in C^*$ and the optimization in $z^*$ in question is quadratic, we may infer that,
$$\tilde{J}^*((v_0^*)_1)=\inf_{z^* \in Y^*}J^*(z^*,(v_0^*)_1)=J^*(z_0^*,(v_0^*)_1).$$
From this, (\ref{us3099}) and (\ref{us3190}), we finally obtain,
\begin{eqnarray}
J(u_0)&=& \min_{u \in U} J(u) \nonumber \\ &=& \max_{v^*_1 \in C^*} \tilde{J}^*(v^*_1) \nonumber \\ &=& \tilde{J}^*((v_0^*)_1)
\nonumber \\ &=& J^*(z_0^*,(v_0^*)_1).
\end{eqnarray}
This completes the proof.
\end{proof}

\section{The Existence of a Global Solution for the Ginzburg-Landau System in the Presence of a Magnetic Field}
 In this section we develop a proof of existence of solution for the
 Ginzburg-Landau system in the presence of a magnetic field and concerning potential. We emphasize again that similar models, which are closely relating those of last sections, are addressed in \cite{100,101}.

 We highlight this existence result and the next duality principle have been presented in similar form in  my book, "A Classical Description of Variational Quantum Mechanics
 and Related Models", \cite{500}. For the sake of completeness, we present both the results in details.
 
 Finally, as a previous related existence result we would cite \cite{1401}. 
 
 \begin{thm} Consider the functional $J: U \rightarrow  \mathbb{R}$ where
\begin{eqnarray}J(\phi,\mathbf{A})&=& \frac{\gamma}{2}\int_\Omega|\nabla \phi-i_m \rho \mathbf{A}\phi|^2_2\;dx
\nonumber \\ &&+\frac{\alpha}{4}\int_\Omega |\phi|^4\;dx-\frac{\beta}{2}\int_\Omega |\phi|^2\;dx
\nonumber \\ &&+ \frac{1}{8\pi}\int_{\Omega_1} |\text{ curl} (\mathbf{A})-\mathbf{B}_0|_2^2\;dx,
\end{eqnarray}

where $\Omega, \Omega_1$ are open bounded, simply connected sets such that $$\overline{\Omega} \subset \Omega_1.$$
We assume the boundaries $\partial \Omega$ and $\partial \Omega_1$ to be regular (Lipschitzian).
Here, again $i_m$ denotes the imaginary unit and $\gamma, \alpha, \beta$ and $\rho$ are positive constants.
Also, $$U=W^{1,2}(\Omega; \mathbb{C}) \times L^2(\Omega_1;\mathbb{R}^3).$$

Suppose there exists a minimizing sequence $(\phi_n,\mathbf{A}_n) \subset U$ for $J$ such that
$$\|\phi_n\|_\infty \leq K, \; \forall n \in \mathbb{N}$$ for some $K>0.$

Under such hypotheses, there exists $(\phi_0,\mathbf{A}_0) \in U$ such that
$$J(\phi_0,\mathbf{A}_0) =\min_{(\phi,\mathbf{A}) \in U}\{J(\phi,\mathbf{A})\}.$$
\end{thm}
\begin{proof}

Define
$$\alpha_1=\inf_{(\phi,\mathbf{A}) \in U}\{J(\phi,\mathbf{A})\} \in \mathbb{R}.$$
From the hypotheses,
$$\lim_{n \rightarrow \infty} J(\phi_n,\mathbf{A}_n)=\alpha_1.$$

From the expression of $J$, there exists $K_1>0$ such that
$$\|\text{curl} (\mathbf{A}_n)\|_2^2 \leq K_1, \; \forall n \in \mathbb{N}.$$

Given $(\phi,\mathbf{A}) \in U,$ define $(\phi', \mathbf{A}') \in U$ by
$$\phi'=\phi e^{i_m\rho \varphi},$$
and
$$\mathbf{A}'=\mathbf{A}+\nabla \varphi,$$ where $\varphi$ will be specified in the next lines.

Observe that,

{\footnotesize{\begin{eqnarray}|\nabla \phi'-i_m\rho \mathbf{A}' \phi'|_2
&=& |\nabla(\phi e^{i_m\rho \varphi})-i_m\rho(\mathbf{A}+\nabla \varphi)\phi e^{i_m\rho \varphi}|_2
\nonumber \\ &=& |\nabla \phi e^{i_m\rho \varphi}+\phi i_m \rho e^{i_m\rho \varphi}\nabla \varphi
-i_m \rho\mathbf{A}\phi e^{i_m\rho \varphi}-i_m\rho\phi\nabla \varphi e^{i_m \rho \varphi}|_2 \nonumber \\ &=&
|(\nabla \phi-i_m \rho \mathbf{A}\phi)e^{i_m\rho \varphi}|_2 \nonumber \\ &=&
|\nabla \phi-i_m \rho \mathbf{A}\phi)|_2. \end{eqnarray}}}

Moreover
$$\text{curl}(\mathbf{A}')=\text{curl}(\mathbf{A})+\text{curl}(\nabla \varphi)=\text{curl}(\mathbf{A}).$$

Also
$$|\phi'|=|\phi e^{i_m \rho \varphi}|=|\phi|.$$

From these last calculations, we may infer  the system gauge invariance, that is,
$$J(\phi,\mathbf{A})=J(\mathbf{\phi'},\mathbf{A}').$$

In particular, we shall choose $\varphi \in W^{1,2}(\Omega_1)$  such that
$$div(\mathbf{A}')=div(\mathbf{A})+\nabla^2\varphi=0,$$ and, denoting by $\mathbf{n}$ the outward normal to $\partial \Omega_1$,
$$\mathbf{A}'\cdot \mathbf{n}=\mathbf{A}\cdot\mathbf{n}+\nabla \varphi \cdot \mathbf{n}=0, \text{ on } \partial \Omega_1$$
that is,
$$\nabla^2 \varphi=-div(\mathbf{A}),\; \text{ in } \Omega_1,$$
$$\nabla \varphi \cdot \mathbf{n}=-\mathbf{A}\cdot \mathbf{n},\; \text{ on } \partial \Omega_1.$$

Observe that at first we would have,
$$\inf_{(\phi,\mathbf{A}) \in U} J(\phi,\mathbf{A}) \leq \inf_{(\phi',\mathbf{A}') \in U} J(\phi',\mathbf{A}').$$

However $$J(\phi'_n,\mathbf{A}'_n)=J(\phi_n,\mathbf{A}_n) \rightarrow \alpha_1, \; \text{ as } n
\rightarrow \infty,$$
so that
$$\inf_{(\phi,\mathbf{A}) \in U} J(\phi,\mathbf{A}) = \inf_{(\phi',\mathbf{A}') \in U} J(\phi',\mathbf{A}').$$

From Friedrichs' inequality, we have,
\begin{eqnarray}
K_1^2 &\geq& \|\text{curl} (\mathbf{A}_n)\|_2^2 \nonumber \\ &=&
\|\text{ curl}(\mathbf{A}'_n)\|_2^2+\|div(\mathbf{A}'_n)\|_2^2 \geq K_2\|\mathbf{A}'_n\|_2^2, \; \forall n \in \mathbb{N},
\end{eqnarray}
for some $K_2>0$.

Hence, $$\|\mathbf{A}'_n\|_2 \leq K_3, \forall n \in \mathbb{N}$$
for some $K_3>0.$

We recall that,
$$\|\phi_n'\|_{\infty} = \|\phi_n\|_\infty \leq K, \forall n \in \mathbb{N}.$$

Hence,
\begin{eqnarray}\label{a800}J(\phi'_n,\mathbf{A}'_n)&=& \frac{\gamma}{2}\int_\Omega|\nabla \phi'_n-i_m \rho \mathbf{A}'_n\phi'_n|^2_2\;dx
\nonumber \\ &&+\frac{\alpha}{4}\int_\Omega |\phi'_n|^4\;dx-\frac{\beta}{2}\int_\Omega |\phi'_n|^2\;dx
\nonumber \\ &&+ \frac{1}{8\pi}\int_{\Omega_1} |\text{curl}(\mathbf{A}'_n)-\mathbf{B}_0|_2^2\;dx \nonumber \\ &\geq&
\frac{\gamma}{2}\int_\Omega|\nabla \phi'_n|^2\;dx-\gamma |\rho|\|\phi_n'\|_{\infty}\|\mathbf{A}_n'\|_2 \|\nabla \phi_n'\|_2
\nonumber \\ &&+\frac{\gamma}{2}|\rho|^2\|\mathbf{A}_n'\phi'_n\|_2^2
\nonumber \\ &&+\frac{\alpha}{4}\int_\Omega |\phi'_n|^4\;dx-\frac{\beta}{2}\int_\Omega |\phi'_n|^2\;dx
\nonumber \\ &&+ \frac{1}{8\pi}\int_{\Omega_1} |\text{curl}(\mathbf{A}'_n)-\mathbf{B}_0|_2^2\;dx
\nonumber \\ &\geq&
\frac{\gamma}{2}\|\nabla \phi'_n\|^2_2\;dx-\gamma K K_3 |\rho|\|\nabla \phi_n'\|_2
\nonumber \\ &&+\frac{\gamma}{2}|\rho|^2\|\mathbf{A}_n'\phi'_n\|_2^2
\nonumber \\ &&+\frac{\alpha}{4}\int_\Omega |\phi'_n|^4\;dx-\frac{\beta}{2}\int_\Omega |\phi'_n|^2\;dx
\nonumber \\ &&+ \frac{1}{8\pi}\int_{\Omega_1} |\text{ curl}(\mathbf{A}'_n)-\mathbf{B}_0|_2^2\;dx. \end{eqnarray}

Suppose, to obtain contradiction, there exists  a subsequence $\{n_k\}$ such that
$$\|\nabla \phi'_{n_k}\|_2 \rightarrow +\infty, \text{ as } k \rightarrow \infty.$$

From this and (\ref{a800}) we obtain,
$$J(\phi_{n_k}',\mathbf{A}_{n_k}') \rightarrow  +\infty, \text{ as } k \rightarrow +\infty,$$

which contradicts
$$ J(\phi'_n,\mathbf{A}_n') \rightarrow \alpha_1, \; \text{ as } n \rightarrow +\infty.$$

Therefore, there exists $K_4>0$ such that
$$\|\nabla \phi'_n\|_2 \leq K_4 \in \mathbb{R}^+, \; \forall n \in \mathbb{N}.$$

Hence, from the Rellich- Krondrachov Theorem, there exists $\phi_0 \in W^{1,2}(\Omega; \mathbb{C})$ such
that, up to a not relabeled subsequence,
$$\nabla \phi_n' \rightharpoonup \nabla \phi_0, \text{ weakly in } L^2,$$
and
$$\phi_n' \rightarrow \phi_0, \text{ strongly in } L^2.$$
Also, since
$$\|\text{curl}(\mathbf{A}'_n)\|_2 \leq K_1, \; \forall n \in \mathbb{N},$$
there exists $\mathbf{v}_0 \in L^2(\Omega_1; \mathbb{R}^3)$ such that
$$\text{curl}(\mathbf{A}_n') \rightharpoonup \mathbf{v}_0, \text{ weakly in } L^2(\Omega_1; \mathbb{R}^3).$$

Also, since $$\|\mathbf{A}_n'\|_2 \leq K_4, \forall n \in \mathbb{N},$$ there exists
$$\mathbf{A}_0 \in L^2(\Omega_1; \mathbb{R}^3),$$ such that, up to a not relabeled subsequence,
$$\mathbf{A}_n' \rightharpoonup \mathbf{A}_0, \; \text{ weakly in } L^2(\Omega_1;\mathbb{R}^3).$$

Now fix $$\hat{\phi} \in C_c^\infty(\Omega_1; \mathbb{R}^3).$$

Thus, we have,
\begin{eqnarray}
\langle \mathbf{A}_0, \text{curl}^*(\hat{\phi}) \rangle_{L^2} &=&
\lim_{n \rightarrow \infty}\langle \mathbf{A}_n', \text{curl}^*(\hat{\phi}) \rangle_{L^2}
\nonumber \\ &=&\lim_{n \rightarrow \infty}\langle \text{curl} (\mathbf{A}_n'), \hat{\phi} \rangle_{L^2}
\nonumber \\ &=& \langle \mathbf{v}_0, \hat{\phi} \rangle_{L^2}. \end{eqnarray}

Since $\hat{\phi} \in C_c^\infty(\Omega_1;\mathbb{R}^3)$
is arbitrary, we may infer that
$$\mathbf{v}_0=\text{curl}(\mathbf{A}_0),$$ in distributional sense.

At this point we shall prove that, up to a not relabeled subsequence, we have,
$$\mathbf{A}_n'\phi_n' \rightharpoonup \mathbf{A}_0\phi_0, \; \text{ weakly in } L^2(\Omega; \mathbb{C}^3).$$

Fix $\mathbf{v} \in L^2(\Omega; \mathbb{C}^3)$.  Therefore, up to a not relabeled subsequence, we have that
$$|\phi_n' \mathbf{v}-\phi_0 \mathbf{v}|_2^2 \rightarrow 0, \text{ a.e. in } \Omega.$$

Observe that  $$\|\phi_n'\|_\infty < K, \; \forall n \in \mathbb{N},$$
so that
$$|\phi_n'\mathbf{v}-\phi_0 \mathbf{v}|_2^2 \leq 2K^2|\mathbf{v}|_2^2 \in L^1(\Omega; \mathbb{R}).$$

Thus, from the Lebesgue dominated convergence theorem, we obtain
$$\|\phi_n'\mathbf{v}-\phi_0\mathbf{v}\|_2^2 \rightarrow 0, \text{ as } n \rightarrow \infty.$$

Hence, since $$\mathbf{A}_n' \rightharpoonup \mathbf{A}_0, \; \text{ weakly in } L^2(\Omega_1;\mathbb{R}^3),$$
we have,
\begin{eqnarray}
&&\left|\int_\Omega(\mathbf{A}_n' \cdot \phi_n'\mathbf{v}-\mathbf{A}_0 \cdot \phi_0\mathbf{v})\;dx\right| \nonumber \\ &=&
\left|\int_\Omega(\mathbf{A}_n' \cdot\phi_n'\mathbf{v}-\mathbf{A}_n' \cdot \phi_0\mathbf{v}+\mathbf{A}_n' \cdot \phi_0\mathbf{v}-\mathbf{A}_0 \cdot \phi_0\mathbf{v})\;dx\right| \nonumber \\ &\leq &
\|\mathbf{A}_n'\|_2 \|\phi_n'\mathbf{v}-\phi_0\mathbf{v}\|_ 2+\left|\int_\Omega \mathbf{A}_n' \cdot \phi_0\mathbf{v}-\mathbf{A}_0 \cdot \phi_0\mathbf{v})\;dx\right|\nonumber \\ &\rightarrow&
0, \text{ as } n \rightarrow \infty.\end{eqnarray}

Since $\mathbf{v} \in L^2(\Omega; \mathbb{C}^3)$ is arbitrary, we may infer that
$$\mathbf{A}_n'\phi_n' \rightharpoonup \mathbf{A}_0\phi_0, \text{ weakly in } L^2(\Omega; \mathbb{C}^3).$$

From this we obtain
$$\nabla \phi_n' -i_m \rho \mathbf{A}_n'\phi'_n \rightharpoonup \nabla \phi_0-i_m \rho \mathbf{A}_0 \phi_0,
\text{ weakly in } L^2(\Omega; \mathbb{C}^3),$$
so that
\begin{eqnarray}
\liminf_{n \rightarrow \infty}\left\{\int_\Omega|\nabla \phi_n' -i_m \rho \mathbf{A}_n'\phi'_n|_2^2\;dx\right\}
\geq \int_\Omega|\nabla \phi_0 -i \rho \mathbf{A}_0\phi_0|_2^2\;dx,
\end{eqnarray}

Also, from $$\phi_n' \rightharpoonup \phi_0, \text{ weakly in }W^{1,2}(\Omega; \mathbb{C})$$
and
$$\text{curl}(\mathbf{A}_n') \rightharpoonup \text{curl}(\mathbf{A}_0), \text{ weakly in }L^2(\Omega_1,\mathbb{R}^3),$$
from the convexity of the functional involved, we obtain,
\begin{eqnarray}
&&\liminf_{n \rightarrow \infty} \left\{ \frac{1}{8 \pi} \int_{\Omega_1}|\text{curl}(\mathbf{A}_n')-\mathbf{B}_0|^2_2\;dx
+\frac{\alpha}{4}\int_\Omega |\phi_n'|^4\;dx \right\} \nonumber \\ &\geq&
\frac{1}{8 \pi} \int_{\Omega_1}|\text{curl}(\mathbf{A}_0)-\mathbf{B}_0|^2_2\;dx
+\frac{\alpha}{4}\int_\Omega |\phi_0|^4\;dx,
\end{eqnarray}
so that, from these last results and from $$\phi_n' \rightarrow \phi_0, \text{ strongly in } L^2(\Omega; \mathbb{C}),$$
we get,
\begin{eqnarray}
\inf_{(\phi, \mathbf{A}) \in U}J(\phi,\mathbf{A})&=&\alpha_1 \nonumber \\ &=&
\liminf_{n \rightarrow \infty} J(\phi_n',\mathbf{A}_n') \nonumber \\ &\geq&
J(\phi_0, \mathbf{A}_0).
\end{eqnarray}

The proof is complete.
\end{proof}

\section{Duality for the  Complex Ginzburg-Landau System}
\noindent In this subsection we present a duality principle and relating sufficient optimality criterion for the
full complex Ginzburg-Landau system.

The basic results on convex analysis here developed may be found in \cite{29,12}.
Our results are summarized by the following theorem.

\begin{thm} Let $\Omega,\Omega_1 \subset \mathbb{R}^3$ be  open, bounded, connected sets with  regular (Lipischtzian) boundaries
denoted by $\partial \Omega$ and $\partial \Omega_1$ respectively, where $\overline{\Omega} \subset \Omega_1$ and $\Omega$ corresponds to a super conducting sample. Consider the Ginzburg-Landau energy given by $J:V_1 \times V_2 \rightarrow \mathbb{R}$
where,
\begin{eqnarray}J(\phi,\mathbf{A})&=& \frac{\gamma}{2}\int_\Omega|\nabla \phi-i_m \rho \mathbf{A}\phi|^2_2\;dx
\nonumber \\ &&+\frac{\alpha}{2}\int_\Omega (|\phi|^2-\beta)^2\;dx-\langle \phi,f\rangle_{L^2} \nonumber \\ &&
+ \frac{1}{8\pi}\int_{\Omega_1} |\text{curl} (\mathbf{A})-\mathbf{B}_0|_2^2\;dx,
\end{eqnarray}
and where $\alpha,\;\gamma,\;\rho>0$, $f \in L^2(\Omega;\mathbb{C}).$

In particular, from the Ginzburg-Landau theory for the dimensionless case we have, $\gamma=1,$ $\alpha =\frac{1}{2(1+t^2)^2}$ and $\beta =1-t^4$,
where $t=T/T_c$,  $T_c$ is the critical temperature and $T$ is the super-conducting sample actual one. A typical value for $t$ is $t=0.95$. Finally the value $1/(8\pi)$ may also vary according to type of material or type of superconductor.

Moreover,
$$V_1=W^{1,2}(\Omega;\mathbb{C}),$$
$$V_2=W^{1,2}(\Omega_1;\mathbb{R}^3).$$

Here,  we generically denote
$$\langle g,h\rangle_{L^2}=\int_\Omega Re[g] Re[h]\;dx+\int_\Omega Im[g]Im[h]\;dx,$$
$\forall h,g \in L^2(\Omega;\mathbb{C})$, where $Re[a], Im[a]$ denote the real and imaginary parts of $a$,
$\forall a \in \mathbb{C},$ respectively.

We also denote,
$$J(\phi,\mathbf{A})=G_0(\phi,\nabla \phi,\mathbf{A})+G_1(\phi,0)+G_2(\mathbf{A}),$$
$$G_0(\phi,\nabla\phi,\mathbf{A})=\frac{\gamma}{2}\int_\Omega |\nabla \phi-i_m\rho\mathbf{A}\phi|_2^2\;dx,$$
$$G_1(\phi,v_3)=\frac{\alpha}{2}\int_\Omega (|\phi|^2-\beta+v_3)^2\;dx -\langle \phi,f\rangle_{L^2},$$
and, $$G_2(\mathbf{A})=\frac{1}{8\pi}\int_{\Omega_1} |\text{ curl} (\mathbf{A})-\mathbf{B}_0|_2^2\;dx.$$

Moreover, we define,
\begin{eqnarray}G_0^*(v_1^*)&=&\sup_{(\phi,v_1) \in V_1 \times Y}\{\langle v_1^*, v_1-i_m\rho \mathbf{A} \phi \rangle_{L^2}-G_0(\phi,v_1,\mathbf{A})\}
\nonumber \\ &=& \frac{1}{2 \gamma}\int_\Omega |v_1^*|^2_2\;dx,\end{eqnarray}

\begin{eqnarray}
G_1^*(v_1^*,v_3^*,\mathbf{A})&=& \sup_{(\phi,v_3) \in V_1 \times Y_1} \{\langle v_1^*, \nabla \phi -i_m\rho\mathbf{A} \phi \rangle_{L^2}+
\langle v_3^*,v_3 \rangle_{L^2}  -G_1(\phi,v_3)\}\nonumber \\ &=&\frac{1}{2}\int_\Omega \frac{|div(v_1^*)+i_m\rho\mathbf{A}\cdot v_1^*-f|^2}{2v_3^*}\;dx+\frac{1}{2\alpha}\int_\Omega (v_3^*)^2\;dx \nonumber \\ &&
+\int_\Omega \beta (v_3^*)\;dx,
\end{eqnarray}
if $v^* \in B_1$, where
$$B_1=\{v^* \in Y^*\times Y_1^* \;:\; v_3^*>0 \text{ in }\overline{\Omega}\}.$$

We also denote
\begin{eqnarray}B_2&=&\{v^* \in Y^*\;:\; \nonumber \\ &&
\frac{1}{8 \pi}\int_{\Omega_1}|curl \mathbf{A}|_2^2\;dx-\frac{1}{2}\int_\Omega \frac{|\rho v_1^* \cdot \mathbf{A}|^2}{2\;v_3^*} \;dx > 0,
\nonumber \\ && \forall \mathbf{A} \in D^*, \text{ such that } \mathbf{A} \neq \mathbf{0}\},
\end{eqnarray}
$$D^*=\{ \mathbf{A} \in V_2\;:\; div(\mathbf{A})=0, \text{ in } \Omega_1, \text{ and } \mathbf{A} \cdot \mathbf{n}=0
\text{ on }\partial \Omega_1\},$$ $$C^*=B_1 \cap B_2,$$
and $$Y=Y^*=L^2(\Omega;\mathbb{C}^3) \text{ and } Y_1=Y_1^*=L^2(\Omega).$$
Under such assumptions, we have,
\begin{eqnarray}\inf_{(\phi,\mathbf{A}) \in V_1 \times D^*} J(\phi, \mathbf{A})
&\geq& \sup_{v^* \in C^*}\{\inf_{\mathbf{A} \in D^*}\{J^*(v^*,\mathbf{A})+G_2(\mathbf{A})\}\}
\nonumber \\ &=& \sup_{v^* \in C^*} \tilde{J}^*(v^*),\end{eqnarray}
where $$J^*(v^*,\mathbf{A})=-G_0^*(v_1^*)-G_1^*(v_1^*,v_3^*,\mathbf{A}).$$
and $$\tilde{J}^*(v^*)=\inf_{\mathbf{A} \in D^*}\{J^*(v^*,\mathbf{A})+G_2(\mathbf{A})\}.$$

Moreover, assume there exists a critical point $(v_0^*,\mathbf{A}_0) \in C^* \times D^*$ such that
$$\delta \{J^*(v^*_0,\mathbf{A}_0)+G_2(\mathbf{A}_0)\}=\mathbf{0}.$$

Under such hypotheses, defining
$$\phi_0=\frac{div((v_0^*)_1)+i_m\rho \mathbf{A}_0\cdot (v_0^*)_1-f}{2(v_0^*)_3},$$
we have,
\begin{eqnarray}
J(\phi_0,\mathbf{A}_0)&=&\min_{(\phi,\mathbf{A}) \in V_1\times D^*} J(\phi,\mathbf{A})
\nonumber \\ &=& \sup_{v^* \in C^*}\{\inf_{\mathbf{A} \in D^*}\{J^*(v^*,\mathbf{A})+G_2(\mathbf{A})\}\}
\nonumber \\ &=& \max_{v^* \in C^*} \tilde{J}^*(v^*) \nonumber \\ &=& \tilde{J}^*(v_0^*) \nonumber \\
&=& J^*(v_0^*, \mathbf{A}_0)+G_2(\mathbf{A}_0).
\end{eqnarray}
\end{thm}
\begin{proof}
Observe that
\begin{eqnarray}
-J^*(v^*,\mathbf{A})-G_2(\mathbf{A}) &=& G_0^*(v_1^*)+G_1^*(v^*_1,v_3^*,\mathbf{A})-G_2(\mathbf{A})
\nonumber \\ &\geq& \langle v_1^*, \nabla \phi-i_m\rho \mathbf{A} \phi \rangle_{L^2}-G_0(\phi,\nabla \phi,\mathbf{A})
\nonumber \\ &&-\langle v_1^*, \nabla \phi-i_m\rho \mathbf{A} \phi \rangle_{L^2}+\langle v_3^*,0 \rangle_{L^2}
\nonumber \\ && -G_1(\phi,0)-G_2(\mathbf{A}),
\end{eqnarray}
$\forall \phi \in V_1,\; \mathbf{A} \in D^*,$
that is,
\begin{eqnarray}
&&G_0(\phi,\nabla \phi, \mathbf{A})+G_1(\phi,0)+G_2(\mathbf{A})
\nonumber \\ &\geq& -G_0^*(v_1^*)-G_1^*(v_1^*,v_3^*,\mathbf{A})+G_2(\mathbf{A}),
\end{eqnarray}
so that
\begin{eqnarray}
J(\phi,\mathbf{A}) &\geq& \inf_{\mathbf{A} \in D^*}\{-G_0^*(v_1^*)-G_1^*(v_1^*,v_3^*,\mathbf{A})+G_2(\mathbf{A})\}
\nonumber \\ &=& \inf_{\mathbf{A} \in D^*}\{J^*(v^*,\mathbf{A})+G_2(\mathbf{A})\} \nonumber \\ &=&
\tilde{J}(v^*), \; \forall (\phi, \mathbf{A}) \in V_1\times D^*,\; v^* \in C^*.
\end{eqnarray}

Thus,
\begin{equation}\label{us323}\inf_{(\phi,\mathbf{A}) \in V_1\times D^*} J(\phi,\mathbf{A}) \geq \sup_{v^* \in C^*} \tilde{J}^*(v^*). \end{equation}

Now, suppose $(v_0^*,\mathbf{A}_0) \in C^* \times D^*$ is such that
\begin{equation}\label{br9062}\delta \{J^*(v_0^*,\mathbf{A}_0)+G_2(\mathbf{A}_0)\}=\mathbf{0}.\end{equation}

From the variation in $v_1^*$ we obtain,
\begin{eqnarray}(v_0^*)_1&=&\gamma(\nabla-i\rho \mathbf{A}_0)\left(\frac{div((v_0^*)_1)+i_m\rho\mathbf{A}_0  \cdot (v_0^*)_1-f}{2(v_0^*)_3}\right)
\nonumber \\ &=&  \gamma(\nabla -i_m\rho \mathbf{A}_0)\phi_0,
\end{eqnarray}
so that,
\begin{equation}\label{us324}G_0^*((v_0^*)_1)=\langle (v_0^*)_1, \nabla \phi_0-i_m\rho \mathbf{A}_0 \phi_0 \rangle_{L^2}-G_0(\phi_0,\nabla\phi_0,\mathbf{A}_0).\end{equation}

From the variation in $v_3^*$ we obtain
 $$\frac{(div((v_0^*)_1)+i_m\rho\mathbf{A}_0 \cdot (v_0^*)_1-f)^2}{(2(v_0^*)_3)^2}-\frac{(v_0^*)_3}{\alpha}-\beta=0,$$
 that is,
 $$(v_0^*)_3=\alpha(|\phi_0|^2-\beta),$$
 so that
 \begin{equation}\label{us325}G_1^*((v_0^*)_1,(v_0^*)_3, \mathbf{A}_0)=-\langle (v_0^*)_1,\nabla \phi_0-i_m\rho \mathbf{A}_0 \phi_0 \rangle_{L^2}-G_1(\phi_0,0).\end{equation}
From (\ref{us324}) and (\ref{us325}), we obtain

\begin{eqnarray}\label{us326}&& J^*(v_0^*,\mathbf{A}_0)+G_2(\mathbf{A}_0)\nonumber \\ &=&-G_0^*((v_0^*)_1)- G_1^*((v_0^*)_1,(v_0^*)_3, \mathbf{A}_0)+G_2(\mathbf{A}_0) \nonumber
\\ &=& G_0(\phi_0,\nabla \phi_0,\mathbf{A}_0)+G_1(\phi_0,0)+G_2(\mathbf{A}_0) \nonumber \\ &=& J(\phi_0,\mathbf{A}_0)
\end{eqnarray}

From  $v_0^* \in C^*$ and (\ref{br9062}) we have,
$$J^*(v_0^*,\mathbf{A}_0)+G_2(\mathbf{A}_0)=\tilde{J}^*(v_0^*).$$
From this, (\ref{us323}) and (\ref{us326}) we obtain,
\begin{eqnarray}
J(\phi_0,\mathbf{A}_0)&=&\min_{(\phi,\mathbf{A}) \in V_1\times D^*} J(\phi,\mathbf{A})
\nonumber \\ &=& \sup_{v^* \in C^*}\{\inf_{\mathbf{A} \in D^*}\{J^*(v^*,\mathbf{A})+G_2(\mathbf{A})\}\}
\nonumber \\ &=& \max_{v^* \in C^*} \tilde{J}^*(v^*) \nonumber \\ &=& \tilde{J}^*(v_0^*) \nonumber \\
&=& J^*(v_0^*, \mathbf{A}_0)+G_2(\mathbf{A}_0).
\end{eqnarray}
The proof is complete.
\end{proof}

\section{Conclusion} In the present work,  we have developed  duality principles applicable to a large class of variational non-convex models.

In a second step we have applied such  results to a Ginzburg-Landau type equation. We emphasize  the main theorems here  developed are a kind of generalization
of the main results found in Toland \cite{12a}, published in 1979 and \cite{85,2900}.

Following the approach presented in \cite{85}, we also highlight the duality principles obtained may be applied to non-linear and non-convex  models of plates, shells and elasticity.

Finally, as above mentioned, in the last sections we present a global existence result, a duality principle and respective optimality conditions for the complex Ginzburg-Landau system in superconductivity in the presence of a magnetic field and concerned magnetic potential.


\end{document}